\documentclass{amsart}

\usepackage{amssymb,amsmath,amsfonts,mathtools,latexsym}
\usepackage[pdftex]{hyperref} %coloca barra lateral com indice e links nas referencias do texto
\newtheorem{theorem}{Theorem}[section]
\newtheorem{corollary}[theorem]{Corollary}
\newtheorem{definition}[theorem]{Definition}
\newtheorem{example}[theorem]{Example}
\newtheorem{lemma}[theorem]{Lemma}
\newtheorem{problem}[theorem]{Problem}
\newtheorem{proposition}[theorem]{Proposition}
\newtheorem{remark}[theorem]{Remark}
\newcommand{\bb}{\mathbb}%Mathbb letters
\newcommand{\f}{\mathsf}%operatorname and Sans serif letters
\newcommand{\fr}{\mathfrak}%Mathfrak letters
\newcommand{\eq}{\coloneqq} % :=
\newcommand{\ds}{\displaystyle} %\displaystyle
\usepackage{comment} %utilize "begin{comment} end{comment}" para omitir partes do texto
\date{October 23, 2021}

\begin{document}
\title[K\"{o}the's Problem, Kurosch-Levitzki Problem and Graded Rings]{K\"{o}the's Problem, Kurosch-Levitzki Problem and Graded Rings}

\author[De Fran\c{c}a]{Antonio de Fran\c{c}a}
\address{Department of Mathematics, Federal University of Tocantins, Arraias--TO, Brazil}
\email{mardua13@gmail.com}
\thanks{Supported by CNPq and Capes (Brazil)}

\author[Sviridova]{Irina Sviridova}
\address{Department of Mathematics, University of Bras\'{i}lia, Bras\'{i}lia--DF, Brasil}
\email{I.Sviridova@mat.unb.br}
%\thanks{Partially supported by CNPq Brazil}



\begin{abstract}
Let $\fr{R}$ be an associative ring graded by left cancellative monoid $\f{S}$, and $e$ the neutral element of $\f{S}$. We study the following problem: if $\fr{R}_e$ is nil, then is $\fr{R}$ nil/nilpotent? We have proved that if $\fr{R}_e$ is nil (of bounded index) and $\f{f}$-commutative, then $\fr{R}$ is nil (of bounded index). Later, we have shown that $\fr{R}_e$ being nilpotent implies $\fr{R}$ is nilpotent. Consequently, we have exhibited a generalization of  Dubnov-Ivanov-Nagata-Higman Theorem for the graded algebras case. Furthermore, we have exhibited relations between graded rings and the problems of K\"{o}the and Kurosh-Levitzki. We have proved that graded rings and $\f{f}$-commutative rings provide positive solutions to  these problems.
\end{abstract}

\maketitle

%====================================================================
%							INTRODUCTION
%====================================================================

\section{Introduction}

%\linenumbers %numerar linhas a partir daqui

Let $\f{G}$ be a group and $\fr{A}$ an associative algebra with a $\f{G}$-grading. One of the central problems in the study of graded algebras is to obtain non-graded (ordinary) properties from the analysis of gradings of a  given algebra and vice versa. Assume $\fr{A}=\bigoplus_{g\in \f{G}} \fr{A}_g$, a $\f{G}$-grading on $\fr{A}$, with $\f{G}$ a finite group such that $e$ is its neutral element. In \cite{BergCohe86}, Bergen and Cohen showed that if $\fr{A}_e$ is  a $PI$-algebra, then $\fr{A}$ is also a $PI$-algebra. However, a bound for the degree of the polynomial identity satisfied by $\fr{A}$ was not found. In \cite{BahtGiamRile98}, Bahturin, Giambruno and Riley proved the same result, but, in addition, they gave a bound for the minimal degree of the polynomial identity satisfied by $\fr{A}$. In this sense, in this work, we determine some relations between the graded polynomial identities\footnote{For more details about \textit{graded algebras}, \textit{polynomial identity} and \textit{graded polynomial identity}, see \cite{GiamZaic05}.} and the non-graded polynomial identities of a given ring $\fr{R}$: we have studied the class of graded rings with nil neutral component. 

In \cite{BergIsaa73}, Bergman and Isaacs proved that if a finite solvable group $\f{G}$ acts by automorphisms on a ring $\fr{R}$ without non-zero fixed points, i.e. $\fr{R}^{\f{G}}=\{0\}$, and without $|\f{G}|$-torsion, then $\fr{R}^{|\f{G}|}=\{0\}$. They also proved that if $\f{G}$ is a finite group acting on a ring $\fr{R}$ without $|\f{G}|$-torsion, and $\fr{R}^{\f{G}}$ is nilpotent, then $\fr{R}$ is nilpotent. Other result proved by these authors is that if $\fr{R}$ is a ring graded by a finite cyclic group such that $\fr{R}_e$ is central, then the commutator ideal of $\fr{R}$ is nil. Already in \cite{Khuk93}, Khukhro presents the following result (Corollary 4.3.8 (p. 101)): {\it if a Lie ring admits a regular automorphism of prime order, then it is nilpotent}. Already Makarenko, in \cite{Maka18}, using techniques created by Khukhro in \cite{Khuk93}, showed that given a $\f{G}$-graded associative algebra $\fr{A}$, where $\f{G}$ is a finite group of order $n$, if $\fr{A}_e$ has a nilpotent two-sided ideal of finite codimension in $\fr{A}_e$, then $\fr{A}$ has a homogeneous nilpotent two-sided ideal of nilpotency index bounded by a function on $n$ and of finite codimension.

In this work, we study some problems that are direct implications of affirmation ``$\fr{R}_e$ is nil'' (or ``$\fr{R}_e$ is nilpotent''). Our interest was to find conditions to provide the nilpotency of a given graded ring whose the neutral component is nil. Basically, we have (partially) answered the following question:

	\vspace{0.15cm}
	\noindent{\bf Problem \ref{problem01}}:	Does $\fr{R}_e$ be nil imply that $\fr{R}$ is nil?
	\vspace{0.15cm}

\noindent In this way, we have some natural questions: how to characterize a graded ring (or algebra) whose neutral component is nil (or nilpotent)? Does the nil neutral component provide that the ring (or algebra) is nilpotent? If so, what are the possible limits for its nilpotency index?

Let $\f{S}$ be a left cancellative monoid, i.e. $gh=gt$ implies $h=t$ for any $g,h,t\in\f{S}$. Assume that $e$ is the element neutral of $\f{S}$. Let $\fr{R}$ be an associative ring with a finite $\f{S}$-grading $\Gamma$. In this work, we have studied associative rings $\fr{R}$ with an $\f{S}$-grading whose neutral component $\fr{R}_e$ is nil. In various cases, we have provided upper bounds for the nilpotency index of such rings. Firstly, we have proved that 

	\vspace{0.15cm}
	\noindent{\bf Proposition \ref{3.03}}: Let $\fr{R}$ be a ring with a finite $\f{S}$-grading $\Gamma$ of order $d$, where $\f{S}$ is a left cancellative monoid. If $\fr{R}_e=\{0\}$, then $\fr{R}^{d+1}=\{0\}$.
	\vspace{0.15cm}

Later, in Proposition \ref{3.31}, we have showed that ``\textit{If $|\f{Supp}(\Gamma)|=d$, and $\fr{R}_e$ is a nonzero nil ring (resp. of bounded index), then $\fr{R}$ is an $\f{S}$-nil ring (resp. of bounded index), i.e. for each $a\in\fr{R}_g$, for any $g\in\f{S}$, there exists $n_a\in\bb{N}$ such that $a^{n_a}=0$}''.

Let us define an \textit{$\f{f}$-commutative} ring as being an associative ring $\fr{R}$ such that there exist a semigroup $\fr{S}$ that acts on the left of $\fr{R}$, and a mapping $\f{f}:\fr{R}\times\fr{R}\longrightarrow \fr{S}$ such that $ab-\f{f}(a,b)ba=0$ for any $a,b\in\fr{R}$. Therefore, we prove that 

	\vspace{0.15cm}
	\noindent{\bf Theorem \ref{3.15}}: If $\fr{R}_e$ is nil (of bounded index) and $\f{f}$-commutative, then $\fr{R}$ is nil (of bounded index).
	\vspace{0.15cm}

And, using the same techniques applied in the proof of the theorem above, we have proved that ``\textit{$\fr{R}_e$ is nilpotent iff $\fr{R}$ is nilpotent}'', and we have ensured that $r\leq\f{nd}(\fr{R})\leq r|\f{Supp}(\Gamma)|$ for $\f{nd}(\fr{R}_e)=r>1$ (see Theorem \ref{3.18}).

An important application of our results arises when we relate them to \textit{Dubnov-Ivanov-Nagata-Higman Theorem}, Kurosh-\textit{Levitzki Problem} and \textit{K\"{o}the Problem}. Below, let us present these three problems.

First, let us now introduce the Dubnov-Ivanov-Nagata-Higman Theorem. Under suitable conditions, it ensures the equivalence between nil algebras of bounded degree and nilpotent algebras. In 1953, Nagata proved that any nil algebra of bounded degree over a field of characteristic zero is nilpotent (see \cite{Naga52}). Afterwards, in 1956, Higman generalized the result of Nagata for any field (see \cite{Higm56}). Posteriorly, it was discovered that this result was firstly published in 1943 by Dubnov and Ivanov (see \cite{DubnIvan43}). In \cite{Kuzm75}, Kuzmin exhibited a lower bound for the nilpotency index of a nil algebra of bounded index $\fr{R}$ over a field of characteristic zero. He showed that $\f{nd}(\fr{R})\geq n(n+1)/2$, where $n=\f{nd}_{nil}(\fr{R})$. Later, in \cite{Razm74}, Razmyslov proposed a smaller estimate than that given by Higman in \cite{Higm56}, and hence, in \cite{Razm94}, he proved it. Vaughan-Lee described in \cite{Vaug93} an algorithm for computing finite dimensional graded algebras and applied his algorithm to show that ``\textit{if an associative algebra $\fr{A}$ over a field of characteristic zero satisfies $a^4=0$ for any $a\in\fr{A}$, then $\fr{A}^{10}=\{0\}$}''.

Therefore, we have proved the following theorem that is a generalization of Dubnov-Ivanov-Nagata-Higman Theorem:

\vspace{0.15cm}
\noindent{\bf Theorem \ref{3.24}}: Let $\f{S}$ be a left cancellative monoid, and $\fr{A}$ an associative algebra over a field $\bb{F}$ with a finite $\f{S}$-grading, $\f{char}(\bb{F})=p$. Suppose that $\fr{A}_e$ is a nil algebra of bounded index. If $p=0$ or $p>\f{nd}_{nil}(\fr{A}_e)$, then $\fr{A}$ is a nilpotent algebra.
\vspace{0.15cm}

Now, let us present the Kurosh-Levitzki Problem. In Ring Theory, Kurosh-Levitzki Problem is analogous to a weak version Burnside Problem (in Group Theory). The history of Kurosh-Levitzki problem is preceded by others two problems: Kurosh Problem and Levitzki Problem.

The Kurosh's and Levitzki's problems were proposed in \cite{Kuro41} and  \cite{Levi43}, respectively. In \cite{Kuro41}, Kurosh, in analogy to Burnise Problem (in Group Theory), posed the following problem: \textit{is every algebraic algebra\footnote{An algebra $\fr{A}$ is called \textit{algebraic} (of bounded degree) if any element $a\in\fr{A}$ satisfies a non-trivial equation $a^n+\lambda_{n-1} a^{n-1}+\cdots+\lambda_1a=0$ (for some $n$ fixed) with coefficients in the base field.} a locally finite\footnote{An algebra is called \textit{locally finite} if each finite set of this algebra generates an algebra of finite dimensional.} algebra?} In this same work, Kurosh answered his question positively in the case where the elements of the algebra have degrees not greater than 3. In \cite{Jaco45}, Jacobson's results reduce the study of Kurosh Problem for algebraic algebras of bounded degree to the study of nil algebras of bounded degree. Later, in \cite{Levi43}, Levitzki posed that ``\textit{is every nil ring a nilpotent ring?''}. This problem is known as \textit{Levitzki Problem}. Levitzki, in \cite{Levi46}, proved that ``\textit{each nil ring of bounded index is semi-nilpotent\footnote{A ring is called \textit{semi-nilpotent} if each finite set of elements in this ring generates a nilpotent ring.}''}. Kaplansky proved in \cite{Kapl50} that ``\textit{any algebraic algebra satisfying a polynomial identity is locally finite''}. Note that Dubnov-Ivanov-Nagata-Higman Theorem (see Theorem \ref{teonagatahigman}) provides a positive solution to Kurosh Problem.

In this way, Kurosh-Levitzki Problem is formulated by ``\textit{if $\fr{A}$ is nil and finitely generated (as algebra), then is $\fr{A}$ nilpotent?''}. In \cite{Golo64} and \cite{GoloShaf64}, Golod and Shafarevich gave counterexamples ensuring that the problem is false, in general. For more details about Kurosch-Levitzky Problem, see \cite{Zelm97} and the references given there. In \cite{Kapl48}, Kaplansky gave a positive solution to Kurosh-Levitzki Problem when $\fr{A}$ satisfies a polynomial identity\footnote{In fact, this result is a consequence of Theorem 5, in \cite{Kapl48}, when applied to the results found in the works \cite{Kuro41,Levi43,Levi46,Levi49,Kapl50}.}.

In the works\footnote{The works \cite{Shir57(eng),Shir57.1(eng),Shir58(eng)} are English translations of the published scientific works \cite{Shir57,Shir57.1,Shir58}, respectively, of mathematician A.I. Shirshov.} \cite{Shir57(eng),Shir57.1(eng),Shir58(eng)}, Shirshov gave some positive solutions to Kurosh Problem and Levitzki Problem. In \cite{Shir57(eng)}, Shirshov proved that ``\textit{any alternative\footnote{A ring $S$ is called \textit{alternative} if for any $a,b\in S$ we have $(ab)b=a(bb)$ and $b(ba)=(bb)a$.} nil-rings $S$ of bounded index, without elements of order 2 in the additive group, is locally nilpotent''}, and also, ``\textit{any algebraic alternative algebra $S$ of bounded index over a field $\bb{F}$ of characteristic $\neq2$ is locally finite''}. Already in \cite{Shir57.1(eng)}, Shirshov proved the following theorem: ``\textit{suppose that an associative algebra $\fr{A}$ is generated by elements $x_1,\dots,x_r$. Assume that $\fr{A}$ satisfies a polynomial identity of degree $n$, and every monomial in $\{x_i\}$ of degree $\leq n$ is nilpotent. Then $\fr{A}$ is nilpotent''}. The algebra $\fr{A}$ with assumed assumptions is called \textit{an algebra satisfying identical relations} and was defined by Malcev in \cite{Malc50}. Later, using a new class of rings introduced by Drazin in \cite{Draz57}, Shirshov proved in \cite{Shir58(eng)} that ``\textit{any nil SP-ring\footnote{An \textit{SP-ring} is a ring with strongly pivotal monomial. For more details and a formal definition of the \textit{SP-ring}, see \cite{Draz57}.} is locally nilpotent''}. 
	
%{===========================================}
%In \cite{Rile97}, Riley
%
%In \cite{Rile01}, Riley
%
%In \cite{Samo09}, Samoilov
%{===========================================}

Adding in Theorem \ref{3.15} above the hypotheses ``\textit{$\fr{R}_e$ is finitely generated}'', we have proved that $\fr{R}$ is a nilpotent ring, i.e. we have proved that
	
	\vspace{0.15cm}
	\noindent{\bf Theorem \ref{3.19}}: If $\fr{R}$ is nil, finitely generated and $\f{f}$-commutative, then $\fr{R}$ is nilpotent.
	\vspace{0.15cm}
		
From this, the following result generalizes the previous theorem.	
		
	\vspace{0.15cm}
	\noindent{\bf Theorem \ref{3.20}}: If $\fr{R}_e$ is nil, finitely generated and $\f{f}$-commutative, then $\fr{R}$ is a nilpotent ring.
	\vspace{0.15cm}

In general, the assumptions ``finitely generated'' and ``$\f{f}$-commutative'' in the theorem above are necessary, and to guarantee this, we have exhibited some examples.

Finally, we have exhibited a considerable relation between graded rings and {\it K\"{o}the's Problem}. The K\"{o}the's Problem and Kurosh Problem are related (see \cite{Smok01}). This problem was proposed by G. K\"{o}the in 1930 (see \cite{Koth30}), and since then, this conjecture has been confirmed in some classes of rings, but it still does not have a general solution. K\"{o}the's Problem asks whether the sum of two right nil ideals of a ring is nil, or equivalently, ``\textit{if a ring $\fr{R}$ has no nonzero nil ideals, then does $\fr{R}$ have nonzero one-sided nil ideals?}''. Various mathematicians have studied this problem since 1930. To more details, as well as an overview, about K\"{o}the's Problem, see the works \cite{FishKrem83,Ferr01,Smok01,Smok00} and their references. In \cite{Krem72}, Krempa exhibited some problems related to K\"{o}the's Problem. Moreover, K\"{o}the's conjecture has several different formulations  (see \cite{Smok01}), one of them says that K\"{o}the's Problem is equivalent to ``\textit{for any nil ring $\fr{R}$, the ring of $n\times n$ matrices over $\fr{R}$ is nil''}, which was proved by Krempa in \cite{Krem72} and (independently) by Sands in \cite{Sand73}). In our work, we have proved that ``\textit{K\"{o}the's Problem has a positive solution in the class of all $\f{f}$-commutative rings graded by a left cancellative monoid}'' (see Theorem \ref{3.26}). Furthermore, we have shown that 

\vspace{0.15cm}
\noindent{\bf Theorem \ref{3.29}}: A positive answer to {\bf Problem \ref{problem01}} implies that the K\"{o}the's Problem has a positive solution.
\vspace{0.15cm}

\noindent Equivalently, a counterexample to K\"{o}the's Problem would yield a counterexample to {\bf Problem \ref{problem01}}. 

%====================================================================
%							PRELIMINARIES
%====================================================================

\section{Preliminaries}

In this section, we recall some concepts, although basic, which are important for our study. We also introduce the definitions of ``\textit{ring grading by a monoid}'' and ``\textit{$\f{f}$-commutative ring}''. Here, unless stated otherwise, all the rings $\fr{R}$ are associative rings.

First, let us talk a little bit about (one-sided) \textit{cancellative monoid}, \textit{monoid order} and \textit{quotient monoid}. A monoid $\f{S}$ is a nonempty set together with a binary operation from $\f{S}\times\f{S}$ to $\f{S}$ which is associative and has an identity element.

A monoid $\f{S}$ is said to be \textit{left cancellative} (resp. \textit{right cancellative}) if, for any $g,h,t\in\f{S}$, $g \cdot h = g \cdot t$ (resp. $h \cdot g = t \cdot g$) implies $h = t$. If $\f{S}$ is a right and left cancellative monoid, then we say that $\f{S}$ is a cancellative ring. Note that any group is a cancellative monoid.

Let $\f{S}$ be a monoid and $e$ the identity element of $\f{S}$. Given an element $g\in \f{S}-\{e\}$, if there exists an $m\in\bb{N}$ such that $g^m=e$, then we say that the \textit{order} of $g$ is the smallest number $n\in \bb{N}$ such that $g^n=e$, and we denote $\f{o}(g)=n$. Otherwise, if there is not $m\in \bb{N}$ such that $g^m=e$, then we say that $g$ has infinite order, and we denote $\f{o}(g)=\infty$. Note that when $\f{S}$ is a finite left cancellative monoid, all the elements of $\f{S}$ have finite orders.

\begin{definition}

Let $\fr{R}$ be a ring, $\f{S}$ a monoid. An \textbf{$\f{S}$-grading} on $\fr{R}$ is a decomposition 
$$
\Gamma: \fr{R}=\bigoplus_{g\in\f{S}} \fr{R}_g
$$
that satisfies $\fr{R}_g \fr{R}_h\subseteq \fr{R}_{gh}$ for all $g,h\in\f{S}$, where each $\fr{R}_g$ is an additive subgroup of $\fr{R}$. The \textbf{support} of $\Gamma$ is the set $\f{Supp}(\Gamma)=\{g\in\f{S} : \fr{R}_g\neq0\}$.
\end{definition}

We say that $\Gamma$ is a \textit{finite $\f{S}$-grading} on $\fr{R}$ when $\f{Supp}(\Gamma)$ is a finite set. In this case, we say that $\Gamma$ has \textit{finite support} (or ``$\Gamma$ is finite'').

\begin{definition}
Let $\fr{R}$ be a ring. We say that $\fr{R}$ is \textbf{nilpotent} if there exists an integer $n>0$ such that $a_1 a_2\cdots a_n=0$ for any $a_1, a_2, \dots, a_n\in \fr{R}$. In this case, the \textbf{nilpotency index} of $\fr{R}$, denoted by $\f{nd}(\fr{R})$, is defined as the smallest number $d\in \bb{N}$ such that $\fr{R}^d=\{0\}$.

We say that $\fr{R}$ is a \textbf{nil ring} if for each $a\in\fr{R}$, there exists some $n_a\in\bb{N}$ such that $a^{n_a}=0$. If there exists some integer $n>0$ such that $b^n=0$ for any $b\in\fr{R}$, thus $\fr{R}$ is a \textbf{nil of bounded index}.  In this case, the \textbf{nil index} of $\fr{R}$, denoted by $\f{nd}_{nil}(\fr{R})$, is defined as the smallest number $p\in \bb{N}$ such that $a^p=0$ for all $a\in\fr{R}$.
\end{definition}

Note that any nil ring of bounded index is nil, and any nilpotent ring also is a nil ring of bounded index.

\begin{definition}
Let $\fr{R}$ be an $\f{S}$-graded ring. Then $\fr{R}$ is called \textbf{$\f{S}$-nil} if all its homogeneous components are nil, i.e. for any $g\in \f{S}$, we have that any $a\in \fr{R}_g$ is nilpotent. If there exists $k\in\bb{N}$ such that $a^k=0$ for any homogeneous element $a\in\fr{R}$, then $\fr{R}$ is called \textbf{$\f{S}$-nil of bounded index}.
\end{definition}

Notice that if $\fr{R}$ is nil (of bounded index) and has an $\f{S}$-grading, then necessarily $\fr{R}$ is $
\f{S}$-nil (of bounded index).

Finally, let us consider a generalization of commutativity introducing the notion of $\f{f}$-commutativity.

\begin{definition}\label{1.70}
Consider a semigroup\footnote{ A \textit{semigroup} $\fr{S}$ is a nonempty set together with a binary operation from $\fr{S}\times\fr{S}$ to $\fr{S}$ which is associative.} $\fr{S}$, and an associative ring $\fr{R}$. A \textbf{left action} of $\fr{S}$ on $\fr{R}$ is a mapping $\cdot:\fr{S}\times\fr{R}\longrightarrow\fr{R}$ satisfying 
	\begin{equation}\nonumber
	(\lambda\gamma)\cdot x=\lambda(\gamma \cdot x) \ \mbox{ and } \ \lambda\cdot (xy)=(\lambda \cdot x)y \ ,
	\end{equation}
for any $\lambda,\gamma\in \fr{S}$ and $x,y\in\fr{R}$. This action is called an \textbf{action by semigroup}.
\end{definition}

Consider any application $\f{f}:\fr{R}\times \fr{R}\rightarrow \fr{S}$, and define the \textbf{$\f{f}$-commutator} of $\fr{R}$ by
	\begin{equation}\nonumber
	[a,b]_{\f{f}}=ab-\f{f}(a,b) \cdot ba
	\ ,
	\end{equation}
for any $a,b \in\fr{R}$. 
	
Observe that, given a ring $\fr{R}$ and a subring $P$ of $\fr{R}$, any map $\f{f}$ from $\fr{R}\times\fr{R}$ to $P$ ensures that $[a,b]_{\f{f}} = ab-\f{f}(a,b)ba$, for any $a,b\in\fr{R}$, is an $\f{f}$-commutator of $\fr{R}$.

	\begin{definition}
An associative ring $\fr{R}$ is called an \textbf {$\f{f}$-commutative ring} if there exist a semigroup $\fr{S}$ that acts on the left of $\fr{R}$, and a mapping $\f{f}:\fr{R}\times\fr{R}\longrightarrow \fr{S}$ such that $[a,b]_{\f{f}}=0$ for any $a,b\in\fr{R}$.
	\end{definition}

Note that all commutative rings, anti-commutative rings, and the nilpotent rings of index $2$ belong to the class of all $\f{f}$-commutative rings. An interesting question is whether every ring is $\f{f}$-commutative for some $\f{f}$. To answer this question, we need some tools. In fact, in Example \ref{3.22}, we exhibit a ring which is not $\f{f}$-commutative, for any $\f{f}$. In general, if for any $a,b\in \fr{R}$, the equation $xba=ab$ has a solution in some semigroup $\fr{S}$, which acts on $\fr{R}$ from the left, then $\fr{R}$ belongs to class of all $\f{f}$-commutative rings.

\begin{example}
Given any ring $\fr{R}$, and assuming that 
	\begin{equation}\nonumber
\lambda a=\underbrace{a+\cdots+a}_{\lambda-times} \ , \ \ \gamma a=(-\gamma) (-a)=\underbrace{(-a)+\cdots+(-a)}_{(-\gamma)-times} \ , \  \mbox{ and } \ 0 a=0 a \ ,
	\end{equation}
for any $a\in\fr{R}$ and $\lambda, \gamma\in\bb{Z}$, with $\lambda>0$ and $\gamma<0$, we have that $\bb{Z}$ acts on the left of $\fr{R}$ naturally. We can consider for each $\lambda\in\bb{Z}$ the mapping $\f{\lambda}$ satisfying $\f{\lambda}(a,b)=\lambda$ for any $a,b\in\fr{R}$, and hence, the $\lambda$-commutative $[\ ,\ ]_{\lambda}$ is well defined. In particular, for $\lambda=1$, take $\f{1}(a,b)=1$ for any $a,b\in\fr{R}$, and thus, the $\f{1}$-commutator $[\ ,\ ]_1$ is given by $[a,b]_1=ab-1\cdot ba=ab-ba=[a,b]$, and so $[\ ,\ ]_1=[\ , \ ]$. On the other hand, when $\lambda=0$, take $\f{0}(a,b)=0$ for any $a,b\in\fr{R}$, and hence, the $\f{0}$-commutator $[\ ,\ ]_0$ is given by $[a,b]_0=ab-0\cdot ba=ab$, and so $[\ ,\ ]_0$ is the product of $\fr{R}$.
	\end{example}

%====================================================================
%							MAIN RESULTS
%====================================================================

\section{Graded rings with nil neutral component}

In this section, we present some important results concerning $\f{S}$-graded rings with the nil neutral component. Unless otherwise stated, in this section we denote by $\fr{R}$ an associative ring with an $\f{S}$-grading given by $\Gamma: \fr{R}=\bigoplus_{g\in \f{S}} \fr{R}_g$, where $\f{S}$ is an arbitrary left cancellative monoid. We also assume that $\Gamma$ has a finite support, namely $|\f{Supp}(\Gamma)|=d<\infty$.

Let $\fr{R}$ be an $\f{S}$-graded ring. Note that to prove that $\fr{R}$ is nil/nilpotent, it is sufficient to analyze only products of its homogeneous elements. In fact, given $a_1,a_2,\dots,a_k \in \fr{R}$, we can write $a_i=\sum_{j=1}^d a_{ig_j}$, where $a_{ig_j}\in \fr{R}_{g_j}$ and $\f{Supp}(\Gamma)=\{g_1,\dots,g_d\}$. Hence, we have
\begin{align}\label{3.01}
a_1 a_2 \cdots a_k &=\left(\sum_{j_1=1}^d a_{1g_{j_1}} \right)\left(\sum_{j_2=1}^d a_{2g_{j_2}} \right)\cdots \left(\sum_{j_k=1}^d a_{kg_{j_k}} \right) \\
	&= \sum_{j_1,j_2,\dots,j_k =1}^d a_{1g_{j_1}} a_{2g_{j_2}} \cdots a_{kg_{j_k}} \nonumber .
\end{align}
Therefore, without loss of generality, we study only the products of homogeneous elements in the grading of $\fr{R}$.

\begin{remark}\label{3.02}
 Let $a_1, a_2,\dots, a_n\in\fr{R}$ be  homogeneous elements. Note that if $\f{deg}(a_i a_{i+1}\cdots a_{i+l})\notin \f{Supp}(\Gamma)$ for some $i,l=1,\dots,n$, then $a_1 a_2\cdots a_n=0$, since $\fr{R}$ is an associative ring. From this, put $\f{deg}(a_i)=g_i$ for each $i=1,2,\dots,n$, and consider the subset of $\f{S}$
$$
\Lambda_{(g_1,\dots,g_n)}\eq \{g_{i}g_{i+1}\cdots g_{i+m} : i=1,\dots,n, 0\leq m\leq n-i\}.
$$
If $\Lambda_{(g_1,\dots,g_n)}\nsubseteq \f{Supp}(\Gamma)$, then $a_1 a_2\cdots a_n=0$. Therefore, if $a_1 a_2\cdots a_n \neq0$ for some homogeneous elements $a_1, a_2,\dots, a_n\in\fr{R}$, then $\Lambda_{(g_1,\dots,g_n)}\subseteq \f{Supp}(\Gamma)$, where $g_i=\f{deg}({a_i})$ for all $i=1,\dots,n$. 
\end{remark}

Being $\f{Supp}(\Gamma)=d<\infty$, observe that if $g\in\f{Supp}(\Gamma)$, then either $(\fr{R}_g)^{d+1}=\{0\}$ or $e\in\f{Supp}(\Gamma)$, where $e$ is the neutral element of $\f{S}$. In fact, suppose that $e\notin\f{Supp}(\Gamma)$. By contradiction, suppose also that there exist $a_1, a_2,\dots, a_{d+1}\in \fr{R}_g$ such that $a_1 a_2\cdots a_{d+1}\neq0$. Hence, $\{g,g^2, \dots, g^{d+1}\}\subset \f{Supp}(\Gamma)$, since $\fr{R}$ is an associative ring. But $|\f{Supp}(\Gamma)|=d$, and thus, there exist $1\leq l<t\leq d+1$ such that $g^t=g^l$, and hence, $e=g^{t-l}\in\f{Supp}(\Gamma)$, because $\f{S}$ is a left cancellative monoid and $1\leq t-l\leq d$. From this, we obtain a contradiction. Therefore, for any $g\in\f{Supp}(\Gamma)$, it follows that $(\fr{R}_g)^{d+1}=\{0\}$ when $e\notin\f{Supp}(\Gamma)$.

The following result ensures that any $\f{S}$-graded non-nilpotent ring has necessarily some nonzero homogeneous element of degree $e$.

\begin{proposition}\label{3.03}
Let $\fr{R}$ be a ring with a finite $\f{S}$-grading $\Gamma$, where $\f{S}$ is a left cancellative monoid. If $\fr{R}_e=\{0\}$, then $\fr{R}^{d+1}=\{0\}$, where $d=|\f{Supp}(\Gamma)|$.
\end{proposition}
%----------------------
%		\begin{comment}
\begin{proof}
Suppose that $e\notin \f{Supp}(\Gamma)$, and write $n\eq d+1$. Let us show that $\fr{R}^n=\{0\}$. For this purpose, it is sufficient to prove that $a_1 a_2\cdots a_n=0$ for all homogeneous elements $a_1,a_2,\dots,a_n \in \fr{R}$ (see (\ref{3.01})).

By contradiction, suppose that there exist homogeneous elements $a_1,a_2,\dots,a_n \in \fr{R}$ such that $a_1 a_2 \cdots a_n\neq 0$. Put $\f{deg}(a_i)=g_i$ for  $i=1,\dots, n$, and define $\Lambda\eq \Lambda_{(g_1,g_2, \dots, g_n)}$ (as in Remark \ref{3.02}). Hence, by Remark \ref{3.02}, we have $\Lambda\subseteq \f{Supp}(\Gamma)$, and since $|\f{Supp}(\Gamma)|=d$, it follows that $|\Lambda|\leq d<n$. Notice that $\{g_1, g_1 g_2, \dots, g_1 g_2 \cdots g_n\}\subseteq \Lambda$, and hence, we conclude that there exist $1\leq l<t \leq n$ such that
	\begin{equation}\nonumber
g_1 g_2\cdots g_l=(g_1 g_2\cdots g_{l})g_{(l+1)}\cdots g_t
.	\end{equation} 
Thus, since $\f{S}$ is left cancellative, it follows that $e=g_{l+1}\cdots g_t \in\Lambda\subseteq\f{Supp}(\Gamma)$. This contradicts our assumption. Therefore, we prove that $a_1 a_2 \cdots a_n= 0$ for all homogeneous elements $a_1,a_2,\dots,a_n \in \fr{R}$, and hence, by (\ref{3.01}), we conclude that $\fr{R}^n=\{0\}$. Consequently, $\fr{R}$ is nilpotent of index at most $n=d+1$.
\end{proof}
%		\end{comment}
%----------------------

Besides ensuring that any non-nilpotent $\f{S}$-graded ring has at least one nonzero neutral homogeneous element, when $\fr{R}$ has a finite $\f{S}$-grading $\Gamma$ with a null neutral component, the previous proposition provides an upper bound for the nilpotency index $\f{nd}(\fr{R})$ depending only on the support order of $\Gamma$. Observe that the ring $\fr{R}=SUT_n(\bb{F})$, of the strictly upper triangular matrices of order $n\times n$ over arbitrary field $\bb{F}$, with its naturally $\bb{Z}$-grading, is a nilpotent ring whose nilpotency index is exactly the one given by Proposition \ref{3.03}. Now, consider the ring $\fr{R}=\bb{R}[x]$, of all the real polynomials in one variable. Note $\fr{R}$ has naturally a $\bb{Z}$-grading of the infinite support. The subset  $\tilde{\fr{R}}=\{p(x)\in\fr{R} : p(0)=0\}$ of $\fr{R}$ is a $\bb{Z}$-graded ring (with the $\bb{Z}$-grading induced by the $\bb{Z}$-grading of $\fr{R}$) such that $\tilde{\fr{R}}_0=\{0\}$, but its support is not finite and $(\tilde{\fr{R}})^n\neq\{0\}$ for all $n\in\bb{N}$, since $x^n\in\tilde{\fr{R}}_n$.

%%
%\begin{example}
%Let $\bb{F}$ be an arbitrary field and $n\in \bb{N}$, $n>1$. Consider $\fr{R}=SUT_n(\bb{F})$, the ring of the strictly upper triangular matrices of order $n\times n$ over $\bb{F}$. The family of subspaces $(\fr{R}_\gamma)_{\gamma\in\bb{Z}_n}$, where $\fr{R}_\gamma=\f{span}_\bb{F} \{E_{ij} : \overline{j-i}=\gamma\}$, defines a $\bb{Z}_n$-grading on $\fr{R}$ (called an  {\bf elementary $\bb{Z}_n$-grading}\index{grading!elementary} corresponding to $(\overline{0}, \overline{1}, \dots, \overline{n-1})$), namely $\Gamma$. It is easily seen that $\f{Supp}(\Gamma)=\bb{Z}_n-\{\bar{0}\}$ and $nd(\fr{R})=n=|\f{Supp}(\Gamma)|+1$. Therefore, we conclude that the previous proposition provides a good upper bound for the nilpotency index of graded rings (with a finite support) whose neutral component is zero.
%\end{example}
%%
%On the other hand, the following example shows that the ``finite support'' condition is required, in particular.
%%
%\begin{example}
%Consider the ring $\fr{R}=\R[x]$ of all the real polynomials in one variable. We have that $\fr{R}$ is naturally $\bb{Z}$-graded with the infinite support. Now, consider the subset  $\tilde{\fr{R}}=\{p(x)\in\fr{R} : p(0)=0\}$ of $\fr{R}$. Notice that $\tilde{\fr{R}}$ is a $\bb{Z}$-graded ring (with the $\bb{Z}$-grading induced by the $\bb{Z}$-grading of $\fr{R}$) such that $\tilde{\fr{R}}_0=\{0\}$, but its support is not finite and $(\tilde{\fr{R}})^n\neq\{0\}$ for all $n\in\bb{N}$, since $x^n\in\tilde{\fr{R}}_n$.
%\end{example}

In the proof of the previous proposition, we have used combinatorial arguments. Evidently, the techniques used in Proposition \ref{3.03} can be extended to answer the following question: ``what can we say about $\fr{R}$ when $\fr{R}_e$ is nil?''. Thus, one of the most natural question is the following:

\begin{problem}\label{problem01}
	Given a ring $\fr{R}$ with a finite $\f{S}$-grading, does $\fr{R}_e$ nil imply $\fr{R}$ nil?
\end{problem}

A problem similar to Problem \ref{problem01} for infinite support is not valid, in general. In fact, the following example ensures the existence of an $\f{S}$-graded ring with an infinite support, which is not nil, although its neutral component is nil. However, the Corollaries \ref{3.04} and \ref{3.32} give some conditions for the infinite support case of Problem \ref{problem01} to have a positive solution.

\begin{example}[Theorem 2.7, \cite{Smok00}] 
For every countable field $\bb{K}$ there is an associative nil $\bb{K}$-algebra $N$ such that the polynomial ring in one indeterminate over $N$ (which is naturally $\bb{Z}$-graded with the neutral component equal to $N$) is not nil.
\end{example}

Let us now present some results concerning to Problem \ref{problem01}.

\begin{proposition}\label{3.31}
Let $\f{S}$ be a left cancellative monoid and $\fr{R}$ a ring with an $\f{S}$-grading $\Gamma$ of finite support, namely $|\f{Supp}(\Gamma)|=d$. Suppose that $\fr{R}_e$ is a nonzero nil ring. Then the following items are true:
	\begin{itemize}
		\item[i.] $\fr{R}$ is an $\f{S}$-nil ring;
		\item[ii.] Suppose $\fr{R}_e$ is nil of bounded index, namely $\f{nd}_{nil}(\fr{R}_e)=s$. Then: 
		\begin{enumerate}
			\item[a)] $\left(a_1 a_2 \cdots a_{k_g} \right)^s = 0$ for any $g\in \f{Supp}(\Gamma)$ and $a_1, a_2, \dots, a_{k_g}\in\fr{R}_g$, where ${k_g} \eq \f{min}\{\f{o}(g) , d\}$;
		 	\item[b)] there exists $k\in\bb{N}$ such that $\left(a_1 a_2 \cdots a_{k} \right)^s = 0$ for any homogeneous elements $a_1, a_2, \dots, a_k$ of the same homogeneous degree;%, i.e. for any $g\in\f{Supp}(\Gamma)$, we have that $a^s=0$ for any $a\in\left(\fr{R}_g\right)^k$;
		 	\item[c)] $\fr{R}$ is $\f{S}$-nil of bounded index.
		\end{enumerate}
	\end{itemize}
\end{proposition}
%----------------------
%		\begin{comment}
\begin{proof}
{i.} Firstly, we have that $e\in\f{Supp}(\Gamma)$, since $\fr{R}_e\neq\{0\}$. Without loss of generality we can take any $g\in \f{Supp}(\Gamma)-\{e\}$, since $\fr{R}_e$ is nil. Put $n=\f{min}\{\f{o}(g), d\}$ and consider the subset $\beta=\{g, g^2, \dots, g^n\}$ of $\f{S}$. Notice that if $\beta \nsubseteq \f{Supp}(\Gamma)$, then $a^n=0$ for any $a\in\fr{R}_g$, since $\fr{R}$ is an associative ring. For this reason, we can assume $\beta\subseteq\f{Supp}(\Gamma)$. It follows that either $g^n=e$ or $e\notin \beta$. In fact, $e\in \beta$ implies that $g^r=e$ for some $r\in\{1,\dots,n\}$. By definition of $\f{o}(g)$, we have $\f{o}(g)\leq r$. Thus, $r=n$ and $g^n=e$, since $n\leq \f{o}(g)$ and $r\leq n$. % Otherwise, if $g^s\neq e$, i.e. $s\neq \f{o}(g)$, thus $s=d<\f{o}(g)$, since $s=\f{min}\{\f{o}(g), d\}\leq \f{o}(g)$. Hence, $e\notin\beta$.
%Note that $e\notin \beta$ iff $g^r\neq e$ for all $1\leq r\leq n$.

If $g^n=e$, then for any $x\in\fr{R}_g$, we have $x^n\in \fr{R}_e$, and hence, $x$ is nilpotent, since $\fr{R}_e$ is nil. On the other hand, if $e\notin \beta$, we have that $n=d<\f{o}(g)$ and $\beta$ has $d$ different elements, because $g^r\neq e$ for all $1\leq r\leq n$, $\f{S}$ is left cancellative and $e\notin\beta\subseteq \f{S}$. Hence, $\beta=\f{Supp}(\Gamma)$, since $|\beta|=d$ and $\beta\subseteq\f{Supp}(\Gamma)$. From this, it follows that $e\notin\f{Supp}(\Gamma)$, %in this case, otherwise we would have $\f{o}(g)\leq n$, since $\f{S}$ is left cancellative,
which contradicts the initial claim.

Anyway, we show that any element of $\fr{R}_g$ is nilpotent, for all $g\in \f{S}$. Therefore, $\fr{R}$ is $\f{S}$-nil.

{ii.a)} Since $\fr{R}_e\neq\{0\}$, we have $\f{nd}(\fr{R}_e)>1$. Fix any $g\in \f{Supp}(\fr{R})$. If $g=e$, the result is obvious. Assume that $g\neq e$. Notice that $g^l \notin \f{Supp}(\Gamma)$ for some $l\in \bb{N}$ implies that $(\fr{R}_g)^l=\{0\}$. Define ${k_g} \eq \f{min}\{\f{o}(g) , d\}$ for any $g\in\f{Supp}(\fr{R})$, where $d = |\f{Supp}(\Gamma)|$. Consider $\gamma=\{g,g^2, \dots, g^{k_g}\}$. Let us show that either $k_g=\f{o}(g)$ or $\gamma\nsubseteq \f{Supp}(\Gamma)$. Indeed, suppose that ${k_g}\neq \f{o}(g)$, and hence, $d=k_g<\f{o}(g)$. Since $k_g<\f{o}(g)$, it is easy to see that $e\notin\gamma$, and all elements of the set $\gamma$ are different, because $\f{S}$ is left cancellative. Then $|\gamma|=d=|\f{Supp}(\Gamma)|$, and for this reason, we can conclude that $\gamma\nsubseteq\f{Supp}(\Gamma)$ when ${k_g}\neq \f{o}(g)$, since $e\in \f{Supp}(\Gamma)$, but  $e\notin\gamma$.

Let $a_1, a_2, \dots, a_{k_g}\in\fr{R}_g$. If $k_g=\f{o}(g)$, then $a_1 a_2 \cdots a_{k_g}\in\fr{R}_e$. Otherwise, if $\gamma\nsubseteq \f{Supp}(\Gamma)$, then there exists $1\leq l\leq k_g$ such that $g^l\in\gamma-\f{Supp}(\Gamma)$, and hence, $a_1 a_2\cdots a_{l}=0$, and consequently, $a_1 a_2\cdots a_{k_g}=(a_1 a_2\cdots a_l) a_{l+1}\cdots a_{k_g}=0$.

Furthermore, we show that for any $g\in \f{Supp}(\Gamma)$ and $a_1, a_2, \dots, a_{k_g} \in \fr{R}_g$, where $k_g=\f{min}\{\f{o}(g), |\f{Supp}(\Gamma)|\}$, we have either $a_1 a_2\cdots a_{k_g} \in \fr{R}_e$ or $a_1 a_2\cdots a_{k_g}=0$. Therefore, in any case, we conclude that $\left(a_1 a_2\cdots a_{k_g} \right)^s=0$, since $\f{nd}_{nil}(\fr{R}_e)=s$.

{ii.b)} By the arguments of {(ii.a)}, it is sufficient to take $k\eq \f{lcm} \{k_g : g\in\f{Supp}(\Gamma)\}$, since $\f{Supp}(\Gamma)$ is finite.

{ii.c)} In the item {(ii.b)}, for any $g\in\f{Supp}(\Gamma)$, take $a_1=a_2=\cdots=a_k=a\in\fr{R}_g$, and hence, $a^{ks}=0$. The result follows.
\end{proof}
%		\end{comment}
%----------------------

Notice that when we assume that $\fr{R}_e$ is nil, the previous proposition exhibits consequences only for homogeneous components. From now on, we will show more general results, i.e. not only for homogeneous components.

\begin{lemma}\label{3.05}
Let $\f{S}$ be a left cancellative monoid and $\fr{R}$ a ring with an $\f{S}$-grading $\Gamma$ with a finite support of order $d$. For any integer $r>1$ and any homogeneous elements $a_1, a_2, \dots, a_{rd}\in\fr{R}$,  we have that either $a_1 a_2 \cdots a_{rd}=0$ or there exist $0\leq s_0< s_1 < \cdots < s_r \leq rd$ satisfying
	\begin{align}\label{3.06}
e= \f{deg}(a_{s_0+1} \cdots a_{s_1}) = \f{deg}(a_{s_1+1} \cdots a_{s_2}) = \cdots = \f{deg}(a_{s_{r-1}+1} \cdots a_{s_r})
.	\end{align}
\end{lemma}
%----------------------
%		\begin{comment}
\begin{proof}
By Proposition \ref{3.03},  if $e\notin\f{Supp}(\Gamma)$, then $\fr{R}^{d+1}=\{0\}$. From this, the result follows, since $d+1\leq dr$ for all $r>1$ in $\bb{N}$. In this case, observe that we have the first alternative.

Now, assume that $e\in\f{Supp}(\Gamma)$. Suppose that there exist homogeneous elements $a_1, a_2, \dots, a_{rd}\in\fr{R}$ such that $a_1 a_2 \cdots a_{rd}\neq 0$. Let us show that there exist $0\leq s_0< s_1 < \cdots < s_r \leq rd$ such that (\ref{3.06}) holds. Put $\f{deg}(a_i)=g_i$ for each $i=1,2,\dots, rd$. For all $1\leq l\leq k\leq rd$, define $b_{l,k}=a_l a_{l+1}\dots  a_k$, $b_k=b_{1,k}$, and $b_{l,l}=a_l$. It is easy to see that  
	\begin{equation}\label{3.07}
\f{deg}(b_{l,k})=\f{deg}(a_l) \f{deg}(a_{l+1})\cdots \f{deg}(a_k)= g_l g_{l+1} \cdots  g_k ,
	\end{equation}
for all $1\leq l \leq  k\leq rd$. Since $a_1 a_2\dots a_{rd}\neq0$, it follows that $\Lambda\eq \{\f{deg}(b_{l,k}) : 1\leq l \leq k\leq rd \}=\Lambda_{(g_1, g_2,\dots, g_{rd})}$ is contained in $\f{Supp}(\Gamma)$ (see Remark \ref{3.02}). Now, consider the subset $\tilde{\Lambda} \eq \{\f{deg}(b_i) : i=1,2,\dots, rd\}$ of $\Lambda$, and notice that 
\begin{equation}\label{3.08}
|\tilde{\Lambda}|\leq \left\{
\begin{array}{c l}	
     d-1 ,& \mbox{if } e\notin\tilde{\Lambda} \\
     d ,& \mbox{if } e\in\tilde{\Lambda}
 \end{array}\right. \ ,
\end{equation}
since $\tilde{\Lambda} \subseteq \f{Supp}(\Gamma)$, $e\in\f{Supp}(\Gamma)$ and $|\f{Supp}(\Gamma)|=d$. For each $g\in\tilde{\Lambda}$, consider the integer $\lambda_g \eq |\{i: \f{deg}(b_i)=g\}|$, and assume $\lambda_g=0$ for any $g\notin\tilde{\Lambda}$. Take $g_0\in\tilde{\Lambda}$ such that $\lambda_{g_0}= \f{max}\{\lambda_g : g\in \tilde{\Lambda}, g\neq e\}$. Let us show that either $\lambda_e\geq r$ or %there exists at least one $g_0\in \tilde{\Lambda}-\{e\}$ such that 
$\lambda_{g_0}\geq r+1$. 

Firstly, note that $\{i: \f{deg}(b_i)=g\}\cap \{j: \f{deg}(b_j)=h\}=\emptyset$ for any $g\neq h$, and hence, 
	\begin{equation}\label{3.09}
rd = \left|\bigcup_{g\in \tilde{\Lambda}} \{i: \f{deg}(b_i)=g\}\right| =\sum_{g\in \tilde{\Lambda}} |\{i: \f{deg}(b_i)=g\}| =\sum_{g\in \tilde{\Lambda}}\lambda_g
\ .	\end{equation}
Then, by (\ref{3.08}) and (\ref{3.09}), we have 
\begin{equation}\label{3.10}
	rd=\lambda_e + \sum_{g\in \tilde{\Lambda}-\{e\}}\lambda_g \leq \lambda_e + \sum_{g\in \tilde{\Lambda}-\{e\}} \lambda_{g_0} \leq \lambda_e + (d-1)\lambda_{g_0} 
\ .\end{equation}

If $e\notin\tilde{\Lambda}$, then $\lambda_e=0$, and hence, by (\ref{3.10}), it follows that $rd\leq (d-1) \lambda_{g_0}$, which implies that $\lambda_{g_0}>r$. 

Suppose now that $e\in\tilde{\Lambda}$. Assume that $\lambda_g < (r+1)$ for any $g\in \tilde{\Lambda}-\{e\}$. Hence, $\lambda_{g_0}\leq r$, and by (\ref{3.10}), we have $rd \leq \lambda_e + (d-1)r$, 
and thus, $\lambda_e\geq rd-(d-1)r=r$. %Now, assume that $\lambda_e<r$. Again, by (\ref{3.10}), we have that $rd \leq \lambda_e + (d-1)\lambda_{g_0}<r+ (d-1)\lambda_{g_0}$. 
%It follows that $r<\lambda_{g_0}$. From this, there exists $g_0\in\tilde{\Lambda}-\{e\}$ such that $\lambda_{g_0}=\lambda_0\geq r+1$. 
From this, we deduce that $\lambda_{g_0}< (r+1)$ implies $\lambda_e\geq r$.
	
Therefore, we show that either $\lambda_e\geq r$ or there exists at least one $g_0\in \tilde{\Lambda}-\{e\}$ such that $\lambda_{g_0}\geq r+1$. 
	 
Suppose that $\lambda_e\geq r$. Take $1\leq i_1< \dots < i_r \leq rd$ such that $\f{deg}(b_{i_j})=e$ for all $j=1,\dots, r$. Hence, it follows that   
	\begin{align}\label{3.11}
b_{i_r} =& a_1 a_2\cdots  a_{i_r}  \nonumber\\
	   =&(a_1 a_2\cdots  a_{i_1})(a_{(i_1+1)} a_{(i_1+2)}\cdots a_{i_2}) \cdots (a_{(i_{r-1}+1)} a_{(i_{r-1}+2)}\cdots  a_{i_r}) \\
	   =& b_{i_1} b_{(i_1+1), i_2} \cdots b_{(i_{r-1}+1), i_r} . \nonumber
	\end{align}
By (\ref{3.07}), observe that $b_{i_j}=b_{i_{j-1}} b_{(i_{j-1}+1), i_j}$ for all $j=2,\dots, r$, we deduce from (\ref{3.11}) that $e=\f{deg}(b_{i_1})=\f{deg}( b_{(i_1+1), i_2})= \cdots =\f{deg}(b_{(i_{r-1}+1), i_r})$. Thus, putting $s_0=0$, $s_j=i_j$ for $j=1,\dots,r$, we obtain (\ref{3.06}).

Finally, assume that $\lambda_{g_0}\geq r+1$ for some $g_0\in \tilde{\Lambda}-\{e\}$. Take $1\leq i_1<\cdots<i_r<i_{(r+1)}\leq rd$ such that $g_0=\f{deg}(b_{i_1})=\cdots=\f{deg}(b_{i_r})=\f{deg}(b_{i_{(r+1)}})$. Similarly to previous case, we have $b_{i_{(r+1)}}=b_{i_1} b_{(i_1+1), i_2} \cdots b_{(i_{r-1}+1), i_r} b_{(i_{r}+1), i_{(r+1)}}$, and hence, we conclude that $\f{deg}(b_{(i_1+1), i_2})= \cdots =\f{deg}(b_{(i_{r-1}+1), i_r})=\f{deg}( b_{(i_{r}+1), i_{(r+1)}})=e$, since $\f{deg}(b_{i_l})=\f{deg}(b_{i_{l+1}})=g_0$, $b_{i_{l+1}}=b_{i_l}b_{i_l+1 , i_{l+1}}$, for all $l=1,\dots,r$, and $\f{S}$ is left cancellative. Therefore, we obtain (\ref{3.06}) for $s_0=i_1, s_1=i_2, \dots, s_r=i_{r+1}$.
\end{proof}
%		\end{comment}
%----------------------

\begin{remark}\label{3.12}
Under the conditions of Lemma \ref{3.05}, consider any integer $r>1$. Suppose integers $0\leq s_0<s_1<\cdots< s_r\leq rd$ such that (\ref{3.06}) holds. Consider the set $\xi=\{i\in\{1,\dots,r\} : s_i-s_{i-1} > 2d\}$. We have
\begin{equation}\label{3.23}
	\begin{split}
rd &= s_0+\sum_{i=1}^r(s_i-s_{i-1}) +(rd-s_r) \geq \sum_{i=1}^r(s_i-s_{i-1}) \\
& \geq \sum_{i\in\xi}(s_i-s_{i-1})\geq \sum_{i\in\xi} (2d+1)\geq |\xi|(2d+1) \ .
	 \end{split}
\end{equation}
Consider the integer $\hat{r}\in\bb{Z}$, $\hat{r}\geq1$, such that $r\in\{2\hat{r}, 2\hat{r}+1\}$. Observe that $s_i-s_{i-1}\leq 2d$ for at least $\hat{r}+1$ integers $i\in\{1, \dots, r\}$, that is, $r-|\xi|\geq \hat{r}+1$. In fact, firstly suppose $r=2\hat{r}$. Let us show that $|\xi|<\hat{r}$. By contradiction, suppose that $|\xi|\geq \hat{r}\geq1$. By (\ref{3.23}), it follows that
\begin{equation}\nonumber
\begin{split}
rd\geq|\xi|(2d+1)\geq \hat{r} (2d+1) \geq 2\hat{r}d+ \hat{r}\geq rd+1 \ ,
	 \end{split}
\end{equation}
and hence, we obtain a contradiction.

Now, suppose $r=2\hat{r}+1$. By contradiction, assume that $|\xi|\geq \hat{r}+1\geq1$. Again by (\ref{3.23}), we have that
\begin{equation}\nonumber
\begin{split}
rd &\geq|\xi|(2d+1)\geq (\hat{r}+1) (2d+1) \\
&\geq 2\hat{r}d+\hat{r}+2d+1=(2\hat{r}d+d)+1+\hat{r}+d\\
&\geq rd+1+\hat{r}+d\geq rd+1 \ ,
\end{split}
\end{equation}
which is impossible. Therefore, we conclude that $r-|\xi|\geq \hat{r}+1$, for any integer $r\in\{2\hat{r}, 2\hat{r}+1\}$, for any integer $\hat{r}\geq1$.
\end{remark}

Let us now apply the above lemma and the above remark for $\f{f}$-commutative rings. Let us prove that any ring with a finite $\f{S}$-grading is nil if its neutral component is nil and $\f{f}$-commutative.

Given an $\f{f}$-commutative ring $\fr{R}$, consider any monomials $m_1, m_2, m_3\in\fr{R}$, i.e. $m_i$ is the product of elements of $\fr{R}$. For any $x,y,z, t\in\fr{R}$, we have
\begin{equation}\label{3.13}
	\begin{split}
x m_1 y m_2 z m_3 t &=(\f{f}(x, m_1)m_1) xy m_2 z m_3 t \\
					&=(\f{f}(x, m_1)m_1) (\f{f}(xy, m_2)m_2) xyz m_3 t \\
					&=(\f{f}(x, m_1)m_1) (\f{f}(xy, m_2)m_2) (\f{f}(xyz, m_3) m_3) xyzt 
,	\end{split}
\end{equation}
where $\f{f}(x, m_1)m_1, \f{f}(xy, m_2)m_2, \f{f}(xyz, m_3) m_3\in\fr{R}$. We can write also
\begin{equation}\label{3.14}
	\begin{split}
x m_1 y m_2 z m_3 t &=x(\f{f}(m_1, y)y)m_1 m_2 z m_3 t \\
&=x(\f{f}(m_1, y)y)(\f{f}(m_1m_2, z)z) m_1 m_2 m_3 t 
,	\end{split}
\end{equation}
where $\f{f}(m_1, y)y, \f{f}(m_1m_2, z)z\in\fr{R}$. We will use $(\ref{3.13})$ and $(\ref{3.14})$ to prove Theorems \ref{3.15} and \ref{3.20}.

\begin{theorem}\label{3.15}
Let $\f{S}$ be a left cancellative monoid with the neutral element $e$, and $\fr{R}$ an $\f{S}$-graded ring with a finite support $\Gamma$. If $\fr{R}_e$ is nil and $\f{f}$-commutative\index{ring!$\f{f}$-commutative}, then $\fr{R}$ is nil. In addition, if $\fr{R}_e$ is nil of bounded index, then $\fr{R}$ is nil of bounded index.
\end{theorem}
%----------------------
%		\begin{comment}
\begin{proof}
Let $\Gamma: \fr{R}=\bigoplus_{i=1}^d \fr{R}_{g_i}$ be a finite $\f{S}$-grading on $\fr{R}$ with support given by $\f{Supp}(\Gamma)=\{g_1,g_2,\dots,g_d\}\subseteq\f{S}$. Assume that $\fr{R}_e$ is an $\f{f}$-commutative nil ring. If $e\notin \f{Supp}(\Gamma)$, by Proposition \ref{3.03}, it follows that $\fr{R}^{d+1}=\{0\}$, and the result follows.

Assume that $e\in \f{Supp}(\Gamma)$. Let $a=\sum_{i=1}^d a_{g_i} \in \fr{R}$ be an arbitrary element, with $a_{g_i}\in\fr{R}_{g_i}$. Let us show that $a$ is nilpotent, i.e. there exists $q\in\bb{N}$ such that $a^q=0$. By (\ref{3.01}), it is sufficient to consider only the products of $q$ homogeneous components of $a$. Consider the set 
	\begin{equation}\nonumber
 \Lambda=\{b_1 b_2\cdots b_k : 1\leq k \leq 2d,\  b_1,\dots, b_k\in\{a_{g_1},\dots, a_{g_d}\} \} \ ,
 	\end{equation}
which is finite, and its subset $\tilde{\Lambda}=\{b\in\Lambda : \f{deg}(b)=e\}$. It is clear that $\tilde{\Lambda}\neq\emptyset$ or $b_1b_2\cdots b_{2d}=0$ for any $b_1, b_2,\dots, b_{2d}\in\{a_{g_1}, \dots, a_{g_d} \}$. In fact, by Lemma \ref{3.05}, if $b_1b_2\cdots b_{2d}\neq0$  for some $b_1, b_2,\dots, b_{2d}\in\{a_{g_1}, \dots, a_{g_d} \}$, then there exist $0\leq s_0\leq s_1\leq s_2\leq 2d$ such that $e=\f{deg}(b_{s_0+1}\cdots b_{s_1})=\f{deg}(b_{s_1+1}\cdots b_{s_2})$. We have that $\tilde{\Lambda}$ contains all elements of the neutral degree formed by the products of at most $2d$ elements of the set $\{a_{g_1},\dots, a_{g_d}\}$. Obviously, since $\tilde{\Lambda}$ is finite, $\tilde{\Lambda}$ is contained in $\fr{R}_e$ and $\fr{R}_e$ is nil, we can take $r=\f{min}\{p\in\bb{N} : b^p=0, \forall b\in\tilde{\Lambda}\}$.

Put $n=r|\tilde{\Lambda}|$, and fix any $b_1, b_2, \dots, b_{2nd}\in\{a_{g_1},\dots, a_{g_d}\}$. Let us show that the monomial $m=b_1 b_2 \cdots b_{2nd}$ is equal to zero. 
 
To obtain a contradiction, suppose that $m\neq0$. By Lemma \ref{3.05}, since $m\neq0$, there exist $0\leq s_0<s_1<\cdots< s_{2n}\leq 2nd$ such that 
	\begin{equation}\label{3.16}
c_1=b_{s_0+1} \cdots b_{s_1},\ 
 c_2=b_{s_1+1} \cdots b_{s_2}, 
  \dots,\ 
   c_{2n}=b_{s_{(2n-1)}+1} \cdots b_{s_{2n}}\in\fr{R}_e
\ .	\end{equation}
By Remark \ref{3.12}, there exist $i_1, \dots, i_{n}\in\{1,\dots,2n\}$ satisfying $s_{i_j}-s_{i_j-1}\leq 2d$ for all $j\in\{1,\dots,n\}$, and hence, put $\tilde{c}_k=c_{i_k}$ for all $k=1,\dots, n$. Observe that
	\begin{equation}\label{3.27}
\tilde{c}_1=b_{s_{i_1-1}+1} \cdots b_{s_{i_1}},\
 \tilde{c}_2=b_{s_{i_2-1}+1} \cdots b_{s_{i_2}}, 
  \dots,\
   \tilde{c}_{n}=b_{s_{i_n-1}+1} \cdots b_{s_{i_{n}}}\in\tilde{\Lambda}
\ .
	\end{equation}

Since $\tilde{c}_1, \dots, \tilde{c}_{n}\in\tilde{\Lambda}$, and $n=r|\tilde{\Lambda}|$, where $\tilde{\Lambda}$ is a finite set, it follows that there exist $1\leq j_1<j_2< \cdots< j_r\leq n$ such that $\tilde{c}_{j_1}= \tilde{c}_{j_2}= \cdots= \tilde{c}_{j_r}=c$ for some $c\in\tilde{\Lambda}$, and thus, $\tilde{c}_{j_1}\tilde{c}_{j_2} \cdots \tilde{c}_{j_r}=c^r=0$. Now, by (\ref{3.16}) and (\ref{3.27}), we can rewrite $m$ as 
\begin{equation}\nonumber
	\begin{split}
m =& b_1\cdots b_{2nd} =b_1\cdots b_{s_0} (c_1 c_2 \cdots c_{2n}) b_{s_{2n}+1}\cdots b_{2nd} \\
	=& (b_1\cdots b_{s_0} ) (c_1 \cdots c_{i_{j_1}-1})\tilde{c}_{j_1}(c_{i_{j_1}+1} \cdots c_{i_{j_2}-1}) \tilde{c}_{j_2} \cdots \tilde{c}_{j_r} (c_{i_{j_r}+1} \cdots c_{2n})(b_{s_{2n}+1}\cdots b_{2nd})\\
	=& (b_1\cdots b_{s_0} ) m_1 \tilde{c}_{j_1} m_2 \tilde{c}_{j_2} m_3 \cdots m_{r}\tilde{c}_{j_r}m_{r+1} (b_{s_{2n}+1}\cdots b_{2nd}) \ ,
	\end{split}
\end{equation}
where $m_1= (c_1 \cdots c_{i_{j_1}-1})$, $m_2=(c_{i_{j_1}+1}\cdots c_{i_{j_2}-1})$, $\dots$, $m_r=(c_{i_{j_r-1}+1}\cdots c_{i_{j_r}-1})$ and $m_{r+1}=(c_{i_{j_r}+1} \cdots c_{2n})$ belong to $\in\fr{R}_e$. Put $\tilde{m}_1=(b_1\cdots b_{s_0} ) m_1 $, and $\tilde{m}_{r+1}=m_{r+1} (b_{s_{2n}+1}\cdots b_{2nd})$ which not necessarily belong to $\fr{R}_e$. By (\ref{3.13}) and (\ref{3.14}), and since $\fr{R}_e$ is $\f{f}$-commutative, it follows that
\begin{equation}\nonumber
	\begin{split}
m &= \tilde{m}_1 \tilde{c}_{j_1} m_2 \tilde{c}_{j_2} m_3 \cdots m_{r}\tilde{c}_{j_r} \tilde{m}_{r+1}= \tilde{m}_1 (\tilde{c}_{j_1} m_2 \tilde{c}_{j_2} m_3 \cdots m_{r}\tilde{c}_{j_r}) \tilde{m}_{r+1}\\
			&= \tilde{m}_1 \f{f}(\tilde{c}_{j_1}, m_2)m_2 \f{f}(\tilde{c}_{j_1} \tilde{c}_{j_2}, m_3)m_3\cdots \f{f}(\tilde{c}_{j_1} \tilde{c}_{j_2}\cdots \tilde{c}_{j_{r-1}}, m_r)m_r (\tilde{c}_{j_1} \tilde{c}_{j_2}\cdots \tilde{c}_{j_r}) \tilde{m}_{r+1} \ .
	\end{split}
\end{equation}
Since $\tilde{c}_{j_1} \tilde{c}_{j_2}\cdots \tilde{c}_{j_r}=0$, we conclude that $m=b_1\cdots b_{2nd}=0$. Evidently, this is a contradiction. Furthermore, we conclude that $\fr{R}$ is a nil ring.

To prove the second part of this theorem, it is sufficient to take $r=\f{nd}_{nil}(\fr{R}_e)$ and to proceed as in the first part of this proof.
\end{proof}
%		\end{comment}
%----------------------

By the proof of the previous theorem, if $\fr{R}$ is an $\f{S}$-graded ring whose neutral component is nil of bounded index, we can exhibit an upper bound for $\f{nd}_{nil}(\fr{R})$. Indeed, it is easy to see that $\f{nd}_{nil}(\fr{R})\leq 2sd^2\left(\ds\frac{d^{2d}-1}{d-1}\right)$, where $\f{nd}_{nil}(\fr{R}_e)= s< \infty$ and $d=|\f{Supp}(\fr{R})|$, since $|\tilde{\Lambda}|\leq |\Lambda|\leq d+d^2+\cdots+d^{2d}=\ds\frac{d(d^{2d}-1)}{d-1}$.

Observe also that a proof similar to proof of Theorem \ref{3.15} ensures a positive answer to Problem \ref{problem01} in the class of all $\f{f}$-commutative rings. Beside that, Theorem \ref{3.15} provides that Problem \ref{problem01} has a positive solution in the class of all associative rings with a finite grading whose the neutral component belongs to the class of all $\f{f}$-commutative rings.

Notice that we can weaken the definition of an $\f{f}$-commutator and still obtain that the previous theorem is true. In fact, it is sufficient to assume that a semigroup $\fr{S}$ acts on $\fr{R}$ if $(\lambda\gamma)x=\lambda(\gamma x)$ for any $\lambda,\gamma\in \fr{S}$, and $x\in\fr{R}$, and hence, to define $[a,b]_{\f{f}}=ab-(\f{f}(a,b)b)a$ for any $a,b\in\fr{R}$, where $\f{f}$ is a map from $\fr{R}\times\fr{R}$ into $\fr{S}$. Therefore, (\ref{3.13}) and (\ref{3.14}) are still true, and thus, Theorem \ref{3.15} can also be verified in this case.

\begin{theorem}\label{3.18}
Let $\f{S}$ be a left cancellative monoid and $\fr{R}$ a ring with a finite $\f{S}$-grading $\Gamma$ of order $d$. If $\fr{R}_e$ is nilpotent of index $\f{nd}(\fr{R}_e)=r\geq1$, then $\fr{R}$ is a nilpotent ring, such that $r \leq \f{nd}(\fr{R})\leq dr$ for $r>1$, and $r\leq\f{nd}(\fr{R})\leq d+1$ for $r=1$.
\end{theorem}
%----------------------
%		\begin{comment}
\begin{proof}
Firstly, when $\f{nd}(\fr{R}_e)=r=1$, the result follows from Proposition \ref{3.03}.

Now, suppose that $\fr{R}_e$ is a nilpotent ring with $\f{nd}(\fr{R}_e)=r>1$. We will show that $a_1 a_2 \cdots a_{rd}=0$ for any homogeneous elements $a_1, a_2, \dots, a_{rd}\in\fr{R}$ (see (\ref{3.01})), where $d=|\f{Supp}(\Gamma)|$. 

By Lemma \ref{3.05}, suppose that there exist $0\leq s_0< s_1 < \cdots < s_r \leq rd$ satisfying
	\begin{equation}\nonumber
e= \f{deg}(a_{s_0+1} \cdots a_{s_1}) = \f{deg}(a_{s_1+1} \cdots a_{s_2}) = \cdots = \f{deg}(a_{s_{r-1}+1} \cdots a_{s_r})
.	\end{equation}
 Hence, $(a_{s_0+1} \cdots a_{s_1}), (a_{s_1+1} \cdots a_{s_2}), \dots, (a_{s_{r-1}+1} \cdots a_{s_r}) \in \fr{R}_e$, and thus, it follows that $a_{s_0+1} \cdots a_{s_r} \in (\fr{R}_e)^r=\{0\}$. By this reason, we have that $a_1\dots a_{rd}=0$. %Thus, by Lemma \ref{3.05}, for any $a_1,\dots,a_{rd}\in \fr{R}$, we always have that $a_1\dots a_{rd}=0$.
  Therefore, we conclude that $\fr{R}$ is a nilpotent ring with $\f{nd}(\fr{R})\leq dr$ for $r>1$.
\end{proof}
%		\end{comment}
%----------------------

From the Theorem \ref{3.18}, it easily follows that Problem \ref{problem01} has a positive solution in the class of graded rings with finite support and whose the neutral component is a nilpotent ring.

%
%\begin{example}
%		Let $\fr{R}$ be a commutative nilpotent $\bb{F}$-algebra, whose nilpotency index is $\f{nd}(\fr{R})=2p$, $\bb{F}$ an algebraically closed field and $\f{char}(\bb{F})=p>0$. Consider the algebra given by
%		\begin{equation}\nonumber
%		\fr{A}=\left\{\begin{pmatrix}  0 & a_{12} & a_{13}  \\
%		0 &  0 & a_{23} \\ 
%		0 &  0 & (a_{33})^p
%		\end{pmatrix}
%		: a_{ij}\in \fr{R} \right\}.
%		\end{equation}
%Notice that $\fr{A}$ is a subalgebra of $M_3(\fr{R})$, the $\bb{F}$-algebra of $3\times3$ matrices over $\fr{R}$, and $\fr{A}^2\subseteq \f{SUT}_3(\fr{R})$. Now, consider the $\bb{F}$-algebra $M$ such that 
%		\begin{equation}\nonumber
%		M=\left\{\begin{pmatrix}  \fr{A} & 0_{3\times3}  \\
%		0_{3\times3} & \fr{A} 
%		\end{pmatrix}
%		\right\} .
%		\end{equation}
%We have that $M$ is $\bb{Z}_6$-graded with the elementary grading $\Gamma$ defined by $(\overline{0}, \overline{1}, \dots, \overline{5})\in(\bb{Z}_6)^6$, with support of order equal to $3$. It is easy to see that $(M_{\overline{0}})^{2}=\{0\}$, and hence, by Theorem \ref{3.18}, it follows that $M^6=\{0\}$. Observe that $\f{nd}(M)=4$, and hence, $\f{nd}(M)\leq \f{nd}(M_{\overline{0}})|\f{Supp}(\Gamma)| < \f{nd}(M_{\overline{0}})|\bb{Z}_6|$, i.e the previous theorem provides an upper bound better than if we look only at the order of the group. 
%\end{example}

Let us finish this section by presenting some generalizations of the previous results for the non-finite support case. First, let us recall the definition of \textit{quotient monoid}. Given a monoid $\f{S}$, consider a \textit{congruence relation} $\sim$ on $\f{S}$, that is, $\sim$ is an equivalence relation such that $g\sim h$ and $k\sim t$ imply $gk\sim ht$, for any $g,h,k,t \in \f{S}$. For each $g\in\f{S}$, consider the set $[g]_\sim=\{h\in\f{S} : g\sim h\}$. We have that the set $\{[g]_\sim : g\in\f{S}\}$ together with operation ``$\circ$'' defined by $[g]_\sim\circ[h]_\sim=[gh]_\sim$ forms a monoid, called the \textit{quotient monoid} (or \textit{factor monoid}), and denoted $\f{S}/\!\!\sim$. Note that $\f{S}/\!\!\sim$ is left cancellative when $\f{S}$ is left cancellative. Now, let us define a grading by the quotient monoid $\f{S}/\!\!\sim$ (induced by a $\f{S}$-grading). Let $\f{S}$ be a monoid, and $\fr{R}$ a ring with an $\f{S}$-grading $\Gamma$. Considering the quotient monoid $\f{S}/\!\!\sim$, for some a congruence relation $\sim$ on $\f{S}$, the \textit{$\f{S}/\!\!\sim$-grading on $\fr{R}$ induced by $\Gamma$} is defined by $\widetilde{\Gamma}: \fr{R}=\bigoplus_{\bar{g}\in\f{S}/\!\sim} \fr{R}_{\bar{g}}$, where $\fr{R}_{\bar{g}}=\bigoplus_{\substack{{h\in\f{S}}\\{h\sim g}}}\fr{R}_{h}$ for any $\bar{g}\in\f{S}/\!\!\sim$. Note that $\f{Supp}(\widetilde{\Gamma})=\{\bar{g}\in\f{S}/\sim \ : g\in\f{Supp}(\Gamma)\}$, and hence, if $\Gamma$ is finite, thus $\widetilde{\Gamma}$ is finite.

\begin{corollary}\label{3.04}
Let $\f{S}$ be a left cancellative monoid, and $\sim$ a congruence relation on $\f{S}$. Let $\fr{R}$ be a ring with an $\f{S}$-grading $\Gamma$, not necessarily finite. Consider the quotient monoid $\f{S}/\!\!\sim$, and let $\widetilde{\Gamma}: \fr{R}=\bigoplus_{\bar{g}\in\f{S}/\!\sim} \fr{R}_{\bar{g}}$ be the $\f{S}/\!\!\sim$-grading induced by $\Gamma$. Suppose $\widetilde{\Gamma}$ is finite of order $d$. The following assumptions are true:	
	\begin{itemize}
		\item[i)] If $\{h\in\f{S}: h\sim e \}\cap\f{Supp}(\Gamma)=\emptyset$, then $\fr{R}^{d+1}=\{0\}$;
		\item[ii)] If $\fr{R}_{\bar{e}}=\bigoplus_{\substack{{h\in\f{S}}\\{h\sim e}}}\fr{R}_{h}$ is $\f{f}$-commutative and nil (of bounded index), then $\fr{R}$ is nil (of bounded index);
		\item[iii)] $\fr{R}_{\bar{e}}$ is nilpotent iff $\fr{R}$ is nilpotent.
	\end{itemize}
\end{corollary}
%----------------------
%		\begin{comment}
\begin{proof}
Considering $\fr{R}$ with its $\f{S}/\!\!\sim$-grading $\widetilde{\Gamma}$ (induced by $\Gamma$), it is sufficient to apply Proposition \ref{3.03}, Theorem \ref{3.15} and Theorem \ref{3.18}. The result follows.
\end{proof}
%		\end{comment}
%----------------------

%----------------------

It is important to note that, in general, the previous corollary ensures that for a graded ring with a support, not necessarily finite, we can obtain similar results to the first part of this work. In addition, if support of $\Gamma$ is finite, then Corollary \ref{3.04} is a union of Proposition \ref{3.03}, Theorem \ref{3.15} and Theorem \ref{3.18}.%$d\leq|\f{Supp}(\Gamma)|$.
 
% Observe nevertheless that, in Corollary \ref{3.04}, $\fr{R}_{\bar{e}}=\bigoplus_{h\in H}\fr{R}_{h}$, and hence, the initial claim must be true for the major part of $\fr{R}$.
% 
%	\begin{example}
%		The last corollary works even when $\f{Supp}(\Gamma)$ is infinite. In fact, consider $\fr{R}=M_\infty(\bb{F})$, the ring of matrices of infinite order over a field $\bb{F}$, with almost all entries zero. Naturally, $\fr{R}$ is $\bb{N}$-graded of infinite support. Fixed $n\in\bb{N}$, we can consider $\tilde{\fr{R}}=M_\infty(M_n(\bb{F}))\cong \fr{R}$, the ring of matrices in block of infinite order. Take $\hat{\fr{R}}=M_\infty(SUT_n(\bb{F}))\subset M_\infty(M_n(\bb{F}))$ which is a ring, where $SUT_n(\bb{F})$ is the ring of strictly upper triangular matrices of order $n$. We have that $\hat{\fr{R}}$ is $\bb{Z}$-graded with infinite support that does not contain $n\bb{Z}$. Let $\Gamma$ be the $\bb{Z}/n\bb{Z}$-grading on $\hat{\fr{R}}$ induced by $\bb{Z}$-grading on $\fr{R}$, we have $\f{Supp}(\Gamma)= \bb{Z}_n -\{\bar{0}\}$. Obviously, $\hat{\fr{R}}$ is nilpotent of order $n=|\f{Supp}(\Gamma)|+1$.
%	\end{example}

\begin{corollary}\label{3.32}
	%Suppose $\f{Supp}(\Gamma)$ is finite. 
Let $\f{G}$ be a group and $\fr{R}$ a ring with a $\f{G}$-grading $\Gamma$, not necessarily finite. Let $H$ be a normal subgroup of $\f{G}$ and $\overline{\Gamma}: \fr{R}=\bigoplus_{\bar{g}\in\f{G}/H} \fr{R}_{\bar{g}}$ the $\f{G}/H$-grading induced by $\Gamma$. Suppose $\overline{\Gamma}$ is finite of order $d$. The following assumptions are true:	
	\begin{itemize}
		\item[i)] If $H\cap\f{Supp}(\Gamma)=\emptyset$, then $\fr{R}^{d+1}=\{0\}$;
		\item[ii)] If $\fr{R}_{\bar{e}}=\bigoplus_{h\in H}\fr{R}_{h}$ is $\f{f}$-commutative and nil (resp. nil of bounded index), then $\fr{R}$ is nil (resp. nil of bounded index).
		\item[iii)] $\fr{R}_{\bar{e}}=\bigoplus_{h\in H}\fr{R}_{h}$ is nilpotent iff $\fr{R}$ is nilpotent.
	\end{itemize}
\end{corollary}
%----------------------
%		\begin{comment}
\begin{proof}
The result follows from Corollary \ref{3.04}, since all group is a left (and right) cancellative monoid.
\end{proof}
%		\end{comment}

%====================================================================
%----Applications-------------------------
%====================================================================

\section{Main applications}

In this section, we present three applications of the results of the previous section: one of them generalizes the Dubnov-Ivanov-Nagata-Higman Theorem, and the other two show relations between graded rings and Kurosch-Levitzki Problem and K\"{o}the's Problem. As in the previous sections, here, all the rings (and algebras) are associative, not necessarily with unity.

\subsection{Graded Algebras and Dubnov-Ivanov-Nagata-Higman Theorem}

An important consequence of our results arises to generalize {\it Dubnov-Ivanov-Nagata-Higman Theorem}. Here, let us present a generalization of Dubnov-Ivanov-Nagata-Higman Theorem for graded algebras. Unless otherwise stated, we denote by $\bb{F}$ a field, $\f{S}$ a left cancellative monoid, and $\fr{A}$ an associative $\bb{F}$-algebra with a finite $\f{S}$-grading. Below, we present Dubnov-Ivanov-Nagata-Higman Theorem and a theorem due to Razmyslov.

\begin{theorem}[Dubnov-Ivanov-Nagata-Higman, \cite{DubnIvan43,Naga52,Higm56}]\label{teonagatahigman}
Let $\fr{A}$ be an associative algebra over a field $\bb{F}$. Assume $\f{char}(\bb{F})=p$. Suppose $a^n=0$ for any $a\in\fr{A}$. If $p=0$ or $n<p$, then $a_1 a_2 \cdots a_{2^n-1} = 0$ for any $a_1, a_2, \dots, a_{2^n-1}\in\fr{A}$.
\end{theorem}

The following result ensures a lower nilpotency index for a nil algebra over a field of characteristic zero than the previous theorem.

\begin{theorem}[Theorem 33.1, \cite{Razm94}]\label{teorazm}
In any associative algebra $\fr{A}$ over a field of characteristic zero in which $b^n= 0$ is valid for any $b\in\fr{A}$, the equality $a_1 a_2\dots a_{n^2}=0$ is valid for any $a_1, a_2, \dots, a_{n^2}\in\fr{A}$.
\end{theorem}

Finally, let us deduce an immediate consequence from Theorems \ref{3.18}, \ref{teonagatahigman} and \ref{teorazm}. Therefore, we have answered Problem \ref{problem01} for $\f{S}$-graded algebras over a field of characteristic zero, if $\fr{A}_e$ is nil of bounded index.

\begin{theorem}\label{3.24}
Let $\f{S}$ be a left cancellative monoid, $\bb{F}$ a field of characteristic $\f{char}(\bb{F})=p$, and $\fr{A}$ an associative $\bb{F}$-algebra with an $\f{S}$-grading $\Gamma$ of finite support. Suppose $\fr{A}_e$ is a nil algebra of bounded index $s=\f{nd}_{nil}(\fr{A}_e)>1$. If $p=0$ or $p>s$, then $\fr{A}$ is a nilpotent algebra. In addition, for $d=|\f{Supp}(\Gamma)|$, we have
	\begin{itemize}
		\item[i)] if $p>s$, then $\f{nd}(\fr{A})\leq d(2^s -1)$;
		\item[ii)] if $p=0$ , then $\f{nd}(\fr{A})\leq dq$, where $q= \left\{\begin{array}{cl}
	2^s-1, &\mbox{if}\quad s=2,3,4 \\
	s^2, &\mbox{if}\quad s\geq5
	\end{array} \right.
\ .$
	\end{itemize}
If $\f{nd}(\fr{A}_e)=1$, then $\fr{R}$ is nilpotent for any field $\bb{F}$, and $\f{nd}(\fr{A})\leq d+1$.
\end{theorem}
%----------------------
%		\begin{comment}
\begin{proof}
The first part follows directly from Theorem \ref{teonagatahigman} and from Theorem \ref{3.18}. Already the items {\it i)} and {\it ii)} follow from Theorem \ref{teorazm} and again from Theorem \ref{teonagatahigman}, since $2^n-1\leq n^2$ in $\bb{N}$ iff $n=1,2,3,4$.

The case $s=1$ follows from Theorem \ref{3.18} (or Proposition \ref{3.03}).
\end{proof}
%		\end{comment}
%----------------------

Observe that, in Theorem \ref{3.24}, $\fr{A}_e$ is not necessarily $\f{f}$-commutative, and the bound of the nilpotency degree of $\fr{A}$ depends only on the nil index of $\fr{A}_e$ and the support order of the $\f{S}$-grading on $\fr{A}$.

The following corollary is an immediate consequence of Theorem \ref{3.24}.

\begin{corollary}\label{3.28}
	Let $\f{S}$ be a left cancellative monoid and $\fr{A}$ an associative algebra over a field $\bb{F}$ with a finite $\f{S}$-grading of order $d$. If $\fr{A}_e$ is nil, $\f{char}(\bb{F})\neq2,3$ and $s\in\{2,3,4\}$, then $a_1 a_2 \cdots a_{d(2^s-1)}=0$ for any $a_1, a_2, \dots, a_{d(2^s-1)}\in\fr{A}$.
\end{corollary}
%----------------------
		\begin{comment}
\begin{proof}
It follows directly from previous theorem.
\end{proof}
		\end{comment}
%----------------------

%====================================================================
%====================================================================

\subsection{Graded Algebras and Kurosch-Levitzki Problem}

Another important application of our results is associated with \textit{Kurosh-Levitzki Problem}. Let $\fr{A}$ be an associative algebra over a field arbitrary $\bb{F}$. The Kurosh-Levitzki Problem is formulated by ``\textit{if $\fr{A}$ is nil and finitely generated (as algebra), then is $\fr{A}$ nilpotent?}''. In \cite{Kapl48}, Kaplansky proved the following result:

\begin{theorem}[Theorem 5, \cite{Kapl48}]\label{3.30}
Every finitely generated nil algebra satisfying a polynomial identity\footnote{For a definition and properties of the \textit{polynomial identity}, see Chapter 1 of \cite{GiamZaic05}.} is nilpotent.
\end{theorem}

Below, we leave our contribution to Kurosh-Levitzki Problem. We show that any $\f{f}$-commutative ring is a positive solution to Kuroshi-Levitzki Problem.

\begin{theorem}\label{3.19}
Let $\fr{R}$ be an $\f{f}$-commutative finitely generated ring. If $\fr{R}$ is nil, then $\fr{R}$ is a nilpotent ring.
\end{theorem}
\begin{proof}
Suppose that $\fr{R}$ is nil. Let $n\in \bb{N}$ be the smallest number of generators of $\fr{R}$. Fix a set $\beta$ of generators of $\fr{R}$ with $n$ elements. Let $s\in \bb{N}$ be the largest nilpotency index of the elements of $\beta$. By $(\ref{3.01})$, $(\ref{3.13})$ and $(\ref{3.14})$, it is easy to check that $a_1a_2\cdots a_{(s-1)n+1}=0$ for any $a_1, a_2, \dots, a_{(s-1)n+1}\in\fr{R}$. Thus, we conclude that $\fr{R}$ is a nilpotent ring with nilpotency index $s\leq\f{nd}(\fr{R})\leq (s-1)n+1$.
\end{proof}

By proof of Theorem \ref{3.19}, observe that the nilpotency index of $\fr{R}$ is an integer such that $s\leq\f{nd}(\fr{R})\leq (s-1)n+1$, where $n$ is the smallest number of generators of $\fr{R}$, and $s$ is the largest nilpotency index of the elements of a generator set of $\fr{R}$ with $n$ elements.

The Kurosh-Levitzki Problem can be generalized as follows: ``given a ring $\fr{R}$ with a finite $\f{S}$-grading, if $\fr{R}_e$ is nil and finitely generated, then is $\fr{R}$ a nilpotent ring?''. This is the graded version of the Kurosh-Levitzki Problem. Obviously, Theorem \ref{3.18} ensures that any positive solution of the Kurosh-Levitzki Problem provides a positive solution of the graded version of the Kurosh-Levitzki Problem, for example, ``if $\fr{R}_e$ is nil, finitely generated and satisfies a polynomial identity, then $\fr{R}$ is nilpotent'' (it is enough to apply Theorems \ref{3.30} and \ref{3.18}).

In what follows, let us give some conditions for ``the neutral component is nil'' to imply the nilpotency of the graded rings. This problem can be seen as the graded version of Levitzki Problem. Observe that Theorem \ref{3.24} provides a positive solution to the graded version of the Kurosh Problem. Naturally, as in the graded version of the Kurosh-Levitzki Problem,  Theorem \ref{3.18} ensures that any positive solution of Levitzki Problem provides a positive solution of the graded version of the Levitzki Problem.

\begin{proposition}\label{3.17}
	Let $\fr{R}$ be a ring with a finite $\f{S}$-grading $\Gamma$, with $d=|\f{Supp}(\Gamma)|$. If $\fr{R}_e$ is nil of index $2$ and $\f{char}(\fr{R}_e)\neq2$, then $\fr{R}$ is nilpotent with $\f{nd}(\fr{R})\leq 3d$.
\end{proposition}
\begin{proof}
	Given $a, b\in\fr{R}_e$, we have
	\begin{equation}\nonumber
	0=(a+b)^2=a^2+ b^2+ab+ba=ab+ba \ ,
	\end{equation}
	and hence, $ab=-ba$ for any $a, b\in\fr{R}_e$. Now, considering any $a, b, c\in\fr{R}_e$, it follows that
	\begin{equation}\nonumber
	\begin{split}
	0 &=(ab+c)^2 =(ab)^2+c^2+abc+cab=abc+cab\\
	  &=abc+(ca)b=abc-(ac)b=abc-a(cb) =abc-a(-bc) \\
	  &=2abc
	\ ,
	\end{split}
	\end{equation}
and so $abc=0$, since $\f{char}(\fr{R}_e)\neq2$. Therefore, $(\fr{R}_e)^3=0$. By Theorem \ref{3.18}, it follows that $\fr{R}$ is a nilpotent ring with $\f{nd}(\fr{R})\leq 3d$, where $d=|\f{Supp}(\Gamma)|$.
\end{proof}

\begin{theorem}\label{3.20}
Let $\f{S}$ be a left cancellative monoid and $\fr{R}$ a ring with a finite $\f{S}$-grading $\Gamma$. If $\fr{R}_e$ is nil, $\f{f}$-commutative and finitely generated, then $\fr{R}$ is a nilpotent ring. Moreover, if  $\fr{R}_e\neq\{0\}$ and $\{a_1,\dots, a_n\}$ generates $\fr{R}_e$, then $s \leq \f{nd}(\fr{R})\leq d((s-1)n+1)$, where $d=|\f{Supp}(\Gamma)|$ and $s=\f{min}\{m\in\bb{N} : a_i^m=0, i=1,\dots,n\}$. If $\fr{R}_e=\{0\}$, then $1\leq\f{nd}(\fr{R})\leq d+1$.
\end{theorem}
%----------------------
%		\begin{comment}
\begin{proof}
In fact, by Theorem \ref{3.19}, it follows that $\fr{R}_e$ is nilpotent with $s\leq\f{nd}(\fr{R}_e)\leq r$, where $r=(s-1)n+1$, $s$ and $n$ are as in Theorem \ref{3.19}. Thus, by Theorem \ref{3.18}, we conclude that $\fr{R}$ is nilpotent with $s\leq\f{nd}(\fr{R})\leq dr$.

If $\fr{R}_e=\{0\}$, it follows from Proposition \ref{3.03} that $\fr{R}^{d+1}=\{0\}$.
\end{proof}
%		\end{comment}
%----------------------

Note that the upper bound for the nilpotency index obtained in Theorem \ref{3.20} can be smaller than the one given by Theorem \ref{3.24} depending on the number $n$ of generators of $\fr{A}_e$.

The following examples ensures that the assumptions of previous theorems are necessary. The first three examples present graded rings (or algebras) in which the neutral component is not finitely generated. The last example concerns the case $\fr{R}_e$ is not $\f{f}$-commutative.

\begin{example}
If $\fr{R}_e$ can not be finitely generated, the previous theorem does not hold. To see this, a counterexample is given below. Let $\fr{R}={\bb{Z}[x_1, x_2, x_3, \dots]}/{I}$ be the quotient ring of the polynomial ring over $\bb{Z}$ in the variables $x_1, x_2, x_3, \dots$ by its ideal $I=\langle x_1^2, x_2^3, x_3^4, \dots\rangle$, with the trivial grading ($\fr{R}_e=\fr{R}$). We have that $\fr{R}$ is a commutative ring which is nil but it is not nilpotent.
\end{example}

\begin{example}[5. Remark (I), \cite{Naga52}]\label{3.21}
Let $\bb{K}$ be a field of characteristic $p\neq0$. Let $\fr{A}_k$ be the algebra over $\bb{K}$ with the generating elements $x_1,\dots, x_k$ with the fundamental relations $x_i^p= 0$, $x_i x_j= x_j x_i$ for $i,j=1,2, \dots, k$; and put $\fr{A}=\sum_{k=1}^\infty \fr{A}_k$. Then $\fr{A}$ is a commutative algebra which is nil of bounded index, with the trivial grading for any left cancellative monoid $\f{S}$, but $\fr{A}$ is not nilpotent.
\end{example}

\begin{example}[Lemma 8(5.6), \cite{Rege91}]
Let $\f{E}$ be the infinite dimensional Grassmann algebra over a field of characteristic $p\neq0$. Let us consider $\f{E}^*=\f{E}-\{1\}$. Then $\f{E}^*$ satisfies $x^p=0$ for any $x\in\f{E}$, i.e. $\f{E}^*$ is nil of degree $p$. We have that $\f{E}^*$ is a $\bb{Z}_2$-graded ring, such that $\f{E}_0$ is a nil commutative algebra (ring), but $\f{E}^*$ is not nilpotent. 
\end{example}

\begin{example}[Example $1$, \cite{Golo64}]\label{3.22}%Exercise 5, Section 6.3, page 179, in \cite{Cohn05}
Golod exhibits in \cite{Golo64} a construction of a nil ring $\fr{R}$ which is finitely generated but it is not nilpotent (see also Chapter 8 in \cite{Hers05}). Then, by Theorem \ref{3.19}, the ring $\fr{R}$ can not be $\f{f}$-commutative for any semigroup $\fr{S}$ and map $f$. This ring with the trivial grading also gives an example which shows the necessity of the condition ``$\fr{R}_e$ is $\f{f}$-commutative'' to be required in Theorem \ref{3.20}.
\end{example}

Finally, in Corollary \ref{3.04}, we show that the results obtained in Proposition \ref{3.03}, Theorem \ref{3.15} and Theorem \ref{3.18} can be presented for rings with not necessarily finite gradings. Analogously, using a similar idea, we can generalize other results of this work, such as Theorem \ref{3.24}, Corollary \ref{3.28}, Proposition \ref{3.17} and Theorem \ref{3.20}.

\subsection{Graded Rings and K\"{o}the's Problem}

Perhaps our greatest contribution in this work is to relate Problem \ref{problem01} and K\"{o}the's Problem. In \cite{Koth30}, K\"{o}the conjectured that ``\textit{if a ring $\fr{R}$ has no nonzero nil ideals, then $\fr{R}$ has no nonzero one-sided nil ideals}''. This conjecture is known as K\"{o}the's Problem, and is still unsolved in the general case. Here, we have proved a relation between graded rings and K\"{o}the's Problem, as well as we have given a positive solution to K\"{o}the's Problem for the $\f{f}$-commutative rings case.

The result below exhibits some well-known equivalences of K\"{o}the's Problem, which will be basic tools for our study.

\begin{theorem}[Some equivalences of K\"{o}the's Problem, \cite{Smok01}]\label{3.25}
	The following assumptions are equivalent:
	\begin{itemize}
		\item[i)] If a ring has no nonzero nil ideals, then it has no nonzero one-sided ideals (K\"{o}the's conjecture);
		\item[ii)] The sum of two right nil ideals in any ring is nil;
		\item[iii)] For every nil ring $\fr{R}$, the ring of $2\times2$ matrices over $\fr{R}$ is nil;
		\item[iv)] For every nil ring $\fr{R}$, the ring of $n\times n$ matrices over $\fr{R}$ is nil.
	\end{itemize}
\end{theorem}

The K\"{o}the's Problem has been solved positively in some classes of rings, but no answer in the general case. Note that Dubnov-Ivanov-Nagata-Higman Theorem (see Theorem \ref{teonagatahigman}) provides a positive solution to K\"{o}the's Problem for algebras over fields of characteristic zero (and some other positive characteristics).

Below, let us present one more class of rings that provides a positive solution to K\"{o}the's Problem. In what follows, let us consider a special $\bb{Z}_n$-grading on $M_n(\fr{R})$, where $n\in\bb{Z}$ and $\fr{R}$ is a ring, called \textit{elementary $\bb{Z}_n$-grading}, which is defined by 
$$
\Gamma: M_n(\fr{R})=\bigoplus_{\lambda\in\bb{Z}_n} M_\lambda,
$$
where $M_{\lambda}=\{ E_{ij}(a)\in M_n(\fr{R}) : a\in\fr{R}, \overline{j-i}=\lambda \}$ is a subgroup of $(M_n(\fr{R}), +)$, for any $\lambda\in\bb{Z}_n$. Notice that $M_{\overline{0}}=\left\{\sum_{i=1}^{n}E_{ii}(a_i) \ : \ a_1,a_2,\dots, a_n\in \fr{R} \right\}$ it is the neutral component of $\Gamma$.

\begin{theorem}\label{3.26}
The K\"{o}the's Problem has a positive answer for any $\f{f}$-commutative ring.
\end{theorem}
%----------------------
%		\begin{comment}
\begin{proof}
Let $\fr{R}$ be a nil $\f{f}$-commutative ring. Let us show that $M_2(\fr{R})$ is nil. 

Let $\Gamma: M_2(\fr{R})=M_0\oplus M_1$, with $M_0=\left\{\begin{pmatrix}  \fr{R} & 0 \\
0 & \fr{R} 
\end{pmatrix}
\right\}$, and $M_1=\left\{\begin{pmatrix}  0 & \fr{R} \\
\fr{R} & 0 
\end{pmatrix}
\right\}$, the elementary $\bb{Z}_2$-grading on $M_2(\fr{R})$. Since 
	\begin{equation}\nonumber
\begin{pmatrix}  a & 0 \\
0 & b 
\end{pmatrix}
^n=\begin{pmatrix}  a^n & 0 \\
0 & b^n 
\end{pmatrix}
	\end{equation}
for any $a,b\in\fr{R}$ and $n\in\bb{N}$, we have that $M_0$ is nil. Suppose that $f:\fr{R}\times\fr{R} \rightarrow \fr{S}$, where $\fr{S}$ is a semigroup acting on the left of $\fr{R}$. Then define a semigroup 
$\tilde{\fr{S}}=\left\{\begin{pmatrix}  \alpha & 0 \\
0 & \beta
\end{pmatrix} : \alpha,\beta\in\fr{S}
\right\}\subset M_2(\fr{S})$
with the usual product of diagonal of matrices. Observe that $\tilde{\fr{S}}$ acts on $M_0$ from the left naturally:
	\begin{equation}\nonumber
\begin{pmatrix}  \alpha & 0 \\
0 & \beta
\end{pmatrix}
\begin{pmatrix}  a & 0 \\
0 & b 
\end{pmatrix}
=
\begin{pmatrix}  \alpha a & 0 \\
0 & \beta b 
\end{pmatrix}
\ ,
	\end{equation}
for any $\alpha,\beta \in \fr{S}$, and $a,b\in\fr{R}$. Consider the map $\tilde{\f{f}}$ of $M_0\times M_0$ to $\tilde{\fr{S}}$ defined by
\begin{equation}\nonumber
\tilde{\f{f}}\left(\begin{pmatrix}  a & 0 \\
	0 & b 
\end{pmatrix}
, \begin{pmatrix}  c & 0 \\
	0 & d 
\end{pmatrix}
\right)=\begin{pmatrix}  \f{f}(a,c) & 0 \\
	0 & \f{f}(b,d)
\end{pmatrix}
	\ .
\end{equation}
Observe that $M_0$ is $\tilde{\f{f}}$-commutative, and hence, by Theorem \ref{3.15}, it follows that $M_2(\fr{R})$ is a nil ring. By Theorem \ref{3.25}, we conclude that $\fr{R}$ ensures a positive solution to K\"{o}the's Problem, and thus, the result follows.
\end{proof}
%		\end{comment}
%----------------------

Now, suppose that K\"{o}the's Problem has a negative solution. %, that is, there exists a nil ring $\fr{R}$ which does not satisfy K\"{o}the's Problem.
By Theorem \ref{3.25}, it follows that the $2\times2$ matrices ring $M_2(\fr{R})$ is not a nil ring for some nil ring $\fr{R}$. % Since $M_2(\fr{R})$ has a natural $\bb{Z}_2$-graded, 
Considering the elementary $\bb{Z}_2$-grading on $\fr{A}=M_2(\fr{R})$, we have that $\fr{A}_0\cong\fr{R}\times\fr{R}$ is also nil. Thus, we can conclude that $\fr{A}$ is a negative solution to Problem \ref{problem01}.

In what follows, we show a relation between graded rings and K\"{o}the's Problem.

\begin{theorem}\label{3.29}
A positive answer to Problem \ref{problem01} implies that K\"{o}the's Problem has a positive solution. In particular, a positive solution of Problem \ref{problem01} for $\bb{Z}_n$-graded rings, for some $n\in\bb{N}$, implies that K\"{o}the's conjecture is true.
\end{theorem}
%----------------------
%		\begin{comment}
\begin{proof}
Let $\fr{R}$ be a associate ring, $n\in\bb{N}$, and consider the matrix ring $M_n(\fr{R})$ over $\fr{R}$. Let $\Gamma: M_n(\fr{R})=\bigoplus_{\lambda\in\bb{Z}_n} M_{\lambda}$ be the elementary $\bb{Z}_n$-grading on $M_n(\fr{R})$. Since $M_{\overline{0}}=\left\{\sum_{i=1}^{n}E_{ii}(a_i)\ : \ a_1,a_2,\dots, a_n\in \fr{R} \right\}$ and $\left(\sum_{i=1}^{n}E_{ii}(b_i) \right)^s=\sum_{i=1}^{n}E_{ii}(b_i^s)$ for any $b_i\in \fr{R}$ and $s\in\bb{N}$, we have that $M_{\overline{0}}$ is nil iff $\fr{R}$ is nil. Hence, the positive answer of Problem \ref{problem01} for $\bb{Z}_n$-graded rings implies that the item {\it iv)} of Theorem \ref{3.25} is true for any ring $\fr{R}$, and consequently, K\"{o}the's conjecture is true.%It is enough to apply Theorem \ref{3.25}, since $M_n(\fr{R})$, for any ring $\fr{R}$, is naturally graded with its elementary $\bb{Z}_n$-grading, where $M_n(\fr{R})_{\overline{0}}$ is defined by (\ref{3.27}).
\end{proof}
%		\end{comment}
%----------------------

The previous theorem shows a connection between graded rings and K\"{o}the's Problem. More specifically, Problem \ref{problem01} implies K\"{o}the's Problem. But, are K\"{o}the's Problem and Problem \ref{problem01} equivalent? This question is still unanswered in the general case.
		
%\nolinenumbers %numerar linhas até aqui
%====================================================================
%							REFERENCES
%====================================================================

\bibliographystyle{amsplain}%amsalpha}   % this means that the order of references
			    % is dtermined by the order in which the
			    % \cite and \nocite commands appear
%\bibliography{/home/marcos/Dropbox/Vida-Acadêmica/trabalhos/mypapers/ref-mardua}

\end{document}